\documentclass{amsart}
\usepackage{amsmath}
\usepackage{amsfonts}
\usepackage{amssymb}
\usepackage{graphicx}
\usepackage{float, verbatim}
\usepackage{enumerate}
\usepackage{color}
\usepackage{hyperref}

\newcommand{\sz}{Szemer\'edi}
\newcommand{\er}{Erd\H{o}s }

\newcommand{\ud}{\overline{d}}
\newcommand{\uboxd}{\overline{\dim_{\mathrm{B}}}}
\newcommand{\Haus}{\dim_{\mathrm{H}}}
\newtheorem*{thm*}{Theorem}
\newtheorem*{defn*}{Definition}
\newtheorem*{conj*}{Conjecture}

\newtheorem{thm}{Theorem}[section]

\newtheorem{lma}[thm]{Lemma}

\newtheorem{defn}[thm]{Definition}

\newtheorem{prop}[thm]{Proposition}
\newtheorem{conj}[thm]{Conjecture}
\newtheorem{rem}[thm]{Remark}
\newtheorem{ques}[thm]{Question}

\newtheorem{exm}[thm]{Example}

\begin{document}

\title{Erd\H{o}s Semi-groups, arithmetic progressions and Szemer\'edi's theorem}

\author{Han Yu}
\address{Han Yu\\
School of Mathematics \& Statistics\\University of St Andrews\\ St Andrews\\ KY16 9SS\\ UK\\}
\curraddr{}
\email{hy25@st-andrews.ac.uk}
\thanks{}



\date{}

\dedicatory{}

\begin{abstract}
In this paper we  introduce and study a certain type of sub semi-group of $\mathbb{R}/\mathbb{Z}$ which turns out to be closely related to \sz's theorem on arithmetic progressions.
\end{abstract}

\maketitle

\section{two motivating problems}
\sz's theorem on arithmetic progressions is perhaps one of the most interesting topics in mathematics. There are a lot of materials on this topic, in particular see \cite{SZ}, \cite{Fu}, \cite{Go}. One of the reasons for \sz's theorem being popular is that it has several proofs with very different backgrounds. The aim of this paper is to introduce another point of view for \sz's theorem. 

For fractal dimensions and arithmetic structures, there are some recent results, see for example \cite{FY1}, \cite{FY2}. Here we adopt a different but related approach. We will study the set of numbers in $[0,1)$ whose binary digit expansion does not have arbitrarily long arithmetic progressions of positions of digit $1$. 
\begin{defn}\label{DEF1}
	We say $x\in [0,1)$ is Erd\H{o}sian, if the binary expansion of $x$ does not contain arbitrarily long arithmetic progressions of positions of digit $1$. The collection of all Erd\H{o}sian numbers is a subset of $[0,1)$ and we call it the Erd\H{o}s set $E$.
\end{defn}
Here for any $x\in \{0,1\}^\mathbb{N}$, the "positions of digit $1$" is the following subset of $\mathbb{N}$,
\[
A_x=\{n\in\mathbb{N}: x_n=1\}.
\]
In this paper we mainly discuss the arithmetic sumset $E+E=\{a+b: a,b\in E\}$ of $E$ where the addition is taken with $\mod 1$. However, during the study of the Erd\H{o}s set $E$, we found it interesting (and probably harder) to study the product set $EE=\{ab: a,b\in E  \}$. Here consider $E$ as a subset of $[0,1)$ and we take the product in $\mathbb{R}$. We ask here the following question which may be interesting on its own.
 \begin{ques}\label{Q1}
 	Are there any numbers $x,y\in E$ such that $xy\notin E$?
 \end{ques}
We shall discuss the above question further in Section \ref{RE}. In particular we will see that if $EE\subset E$ then the \sz's theorem follows. This gives some motivation of the above question. To study the \er set $E$, we introduce \er semi-groups or ESGs in $\mathbb{R}/\mathbb{Z}$, see Section \ref{seErdos}. In particular we shall prove that $E$ is an ESG. One of our main motivation is the following conjecture which implies \sz's theorem.
\begin{conj}\label{CON1}
	If $E$ is an ESG with full Hausdorff dimension then $E=\mathbb{R}/\mathbb{Z}$. 
\end{conj}
We note here that neither Question \ref{Q1} nor Conjecture \ref{CON1} can be deduced by \sz's theorem.

\section{Preliminaries}
\subsection{Densities of integer sequences}
Let $A\subset\mathbb{N}$ be a sequence of integers and denote
\[
A(n)=\#\{i\in [1,n]: i\in A\},
\]
where we use the notation $\#B$ for the cardinality of set $B$. Now we recall two notions of density for integer sequences.
\begin{defn}
	The upper natural density of $A$ is defined as
	\[
	\overline{d}(A)=\limsup_{n\to\infty} \frac{A(n)}{n}.
	\]
	Similarly, we define the lower natural density by replacing the above $\limsup$ with $\liminf$.
\end{defn}

\begin{defn}
	The upper Banach density of $A$ is defined as
	\[
	d_B(A)=\limsup_{k,M\to\infty} \frac{1}{k}(A(M+k-1)-A(M)).
	\]
\end{defn}

The upper and lower natural densities have their fractal dimension counterparts as upper and lower box dimensions. The lower natural density is often also related to the Hausdorff dimension whereas the upper Banach density is related to the Assouad dimension.

\subsection{The Hausdorff dimension}
In this paper we will mostly work with the Hausdorff and packing dimensions. For other notions of dimensions see \cite{Fa}, \cite{Ma}. Let $n\geq 1$ be an integer. Let $A\subset \mathbb{R}^n$ be a Borel set. For any $s\in\mathbb{R}^+$ and any $\delta>0$ define the following quantity
\[
\mathcal{H}^s_\delta(A)=\inf\left\{\sum_{i=1}^{\infty}(\mathrm{diam} (U_i))^s: \bigcup_i U_i\supset A, \mathrm{diam}(U_i)<\delta\right\},
\]
where $\mathrm{diam}(B)$ denotes the diameter of a set $B\subset\mathbb{R}^n$.
Then the $s$-Hausdorff measure of $A$ is
\[
\mathcal{H}^s(A)=\lim_{\delta\to 0} \mathcal{H}^s_{\delta}(A).
\]
The Hausdorff dimension of $A$ is
\[
\Haus A=\inf\{s\geq 0:\mathcal{H}^s(A)=0\}=\sup\{s\geq 0: \mathcal{H}^s(A)=\infty\}.
\]
\subsection{The packing dimension}
For the proof of Theorem \ref{Thmsum} we also need the notion of packing dimension. Let $A\subset\mathbb{R}$ be bounded and we use $N(A,r)$ to denote the minimum covering number for $A$ with intervals with length $r>0$. Then the upper box dimension of $A$ is defined as
\[
\uboxd A=\limsup_{r\to\infty} -\frac{\log N(A,r)}{\log r}.
\]
The packing dimension of $A$ is defined as follows
\[
\dim_P A=\inf\left\{\sup_{i\in\mathbb{N}}\{\uboxd A_i\}: A\subset \bigcup_{i} A_i\right\}.
\]
We know that in general for any Borel set $A\subset\mathbb{R}$
\[
\Haus A\leq \dim_P A.
\]
For any sequence $A_i\subset\mathbb{R}$ of Borel sets we also have the following equality which is usually known as the countably stability of dimensions,
\[
\Haus (\bigcup_{i} A_i)=\sup_{i} \Haus A_i,
\]
\[
\dim_P (\bigcup_{i} A_i)=\sup_{i} \dim_P A_i.
\]
Another property we shall use is that for any two Borel sets $A,A'\subset\mathbb{R}$,
\[
\Haus (A\times A')\leq \Haus A+\dim_P A'.
\]
This is a special case of \cite[Theorem 8.10]{Ma}. One way we shall use this result is that whenever $\dim_P A=0$ we have
\[
\Haus(A+A)=0 \text{ and } \Haus(AA)=0,
\]
where
\[
A+A=\{a+b: a,b\in A\}, AA=\{ab: a,b\in A  \}.
\]
The result for $A+A$ is clear because $A+A$ can be identified with a certain projection of $A\times A$. The result for $AA$ follows similarly because it is almost equal to
\[
\log \exp (\log A+\log A).
\]
The problem is that $A$ can contain non-positive elements in this case $\log(.)$ is not well defined. This problem can be addressed by decomposing $A$ into $A_{-}=A \cap [-\infty,0)$ and $A_+=A\cap [0,\infty)$ then
\[
AA=A_-A_-\cup A_+A_-\cup A_+A_+.
\]
Another issue is that $f(x)=\log(x)$ is not Lipschitz at $x=0$. In fact we can assume $0\notin A$ since this will not affect $AA$ other than removing the element $0$. Now we consider the decomposition
\[
A_-=\bigcup_{i\in\mathbb{N}}A\cap (-\infty,-2^{-i}], A_+=\bigcup_{i\in\mathbb{N}}A\cap [2^{-i},\infty). 
\]
We see that as a result
\[
AA=\bigcup_{i,j} A_i A_j
\]
where $A_i, A_j$ are both one of the sets in the above decomposition. Assume without loss of generality that $A_i, A_j$ are both contained in $(0,\infty)$. Now $A_i, A_j$ are away from $0$, therefore we can perform the $\log(.)$ function and convert the product set to sum set and the result follows because the Hausdorff and packing dimensions are countably stable.
\subsection{Iterated convolutions}

 Later we shall need an important result in \cite{LMP}.

\begin{thm}\label{ThmLMP}
	Let $E\subset [0,1)$ be a closed $\times p \mod 1$ invariant set. If $\Haus E>0$, then
	\[
	\lim_{n\to\infty} \Haus \left(\sum_{i=1}^n E\right)= 1.
	\]
\end{thm}

Let $T: [0,1)\to [0,1)$ be such that $T(x)=px \mod 1$. Then we say that a set  $E\subset [0,1)$ is $\times p\mod 1$-invariant if $TE\subset E$. From the above theorem we see that if there exists a Borel set $E_{\infty}\subset [0,1)$ such that for all integer $n\geq 1$ we have $\sum_{i=1}^n E\subset E_{\infty}$  then
\[
\Haus E_{\infty}=1.
\]
For general subsets of $[0,1]$ which might not be Borel nor $\times p \mod 1$ invariant, we can still study the sum set and iterated sum set. This topic is discussed in a forthcoming paper \cite{FHY}. 

\subsection{Van der Waerden's theorem}

The next important ingredient is van der Waerden's theorem \cite{VW}. We shall state it in a form which will be used directly.
\begin{thm}[Van der Waerden]\label{VAN}
	Let $A_1,A_2,A_3\subset\mathbb{N}$ be three sequences of integers. Let $i,j,k$ be three integers greater than $2$. Suppose that $A_1$ does not contain any $i$-term arithmetic progressions, $A_2$ does not contain any $j$-term arithmetic progressions and $A_3$ does not contain any $k$-term arithmetic progressions. Then there is a number $W(i,j,k)$ such that $A_1\cup A_2\cup A_3$ does not contain any $W(i,j,k)$-term arithmetic progressions.
\end{thm}

\subsection{Sets defined by digit expansions}

 We will also need the following lemma proved in \cite[Example 1.3.2 and 1.4.2]{BP}.

\begin{lma}\label{lmaDim}
	Let $S=s_1s_2\dots$ be an infinite sequence of $\{0,1\}$ and the positions of digit $1$, when viewed as a subset of $\mathbb{N}$, has lower natural density $\alpha>0$. Let
	\[
	E_S=\left\{x\in [0,1): x=\sum_{i=1}^{\infty} x_i2^{-i}, x_i\in \{0,1\} \text{ if } s_i=1\text{ and } x_i=0 \text{ if } s_i=0\right\}.
	\]
	Then we have
	\[
	\Haus E_S= \alpha.
	\]
\end{lma}

 \section{\er semi-groups}\label{seErdos}
 In this section we introduce \er semi-groups which will be the main topic of this paper. Let $x=x_1x_2\dots$ be an infinite sequence over $\{0,1\}$ and  with a possible abuse of the notation we denote $x$ to be the following real number as well
 \[
 x=\sum_{i=1}^{\infty}2^{-i} x_i.
 \]
 Such an association is not one-one, however, the only two to one situations happen when the real number $x$ is a dyadic rational number.  In this case, $x$ has a binary expansion with only finitely many digits $1$ and we associate $x$ only with this sequence. We will use this identification of $0,1$ sequences with real numbers in $[0,1)$. Given two numbers $x,y\in [0,1)$, we write their binary expansion again as $x,y\in \{0,1\}^{\mathbb{N}}$. We write $x_i$ with $i\in\mathbb{N}$ to be the $i$-th digit of the binary expansion of $x$. We call $y\in\mathbb{R}^+$ a subsequence of $x\in\mathbb{R}^+$ in terms of binary expansion if there exists $N>0$ such that $\forall i\in\mathbb{N},y_{i+N}=1\implies x_i=1$ and $y_{j}=0$ for all $j\leq N-1$.  In this paper we shall identify $\mathbb{R}/\mathbb{Z}$ naturally with $[0,1)$.
 
 \begin{defn}\label{ER}
 	Let $E\subset \mathbb{R}/\mathbb{Z}$ be a set with the following properties:
 	\begin{itemize}
 		\item[1]: $E=\bigcup_{i\in\mathbb{N}} F_i$ where each $F_i$ is a closed $\times 2\mod 1$ invariant set.
 		\item[2]: There is a function $W:\mathbb{N}^2\to\mathbb{N}$ such that $\forall i,j\in\mathbb{N}$, $F_i+F_j\subset F_{W(i,j)}$.
 		\item[3]: For all integer $i$, if $x\in F_i$, then $y\in F_i$ for any subsequence $y$ of $x$ in terms of binary expansion.
 	\end{itemize}
 	We shall call such a set $E$ a binary \er semi-group (or simply an ESG) in $\mathbb{R}/\mathbb{Z}$.
 \end{defn}
 
 It is also possible to define $k$-ary \er semi-groups for $k\geq 3$ but we will not need this generalization in this paper. The first observation is that an ESG is a sub semi-group of $\mathbb{R}/\mathbb{Z}$. Indeed, given $x,y\in E$ we see that there exist integers $i,j$ such that $x\in F_i, y\in F_j$ then $x+y\in F_{W(i,j)}\subset E$. Based on this observation we can show the following result.
 
 \begin{prop}\label{PR01}
 	If $E\subset\mathbb{R}/\mathbb{Z}$ is an ESG, then
 	\[
 	\Haus E\in\{0,1\}.
 	\]
 \end{prop}
 \begin{rem}
 	We can compare this result with another result in \cite{EV} which says that there exist Borel subgroup of $\mathbb{R}/\mathbb{Z}$ with any possible Hausdorff dimension in between $0$ and $1$.
 \end{rem}
 \begin{proof}
 	If for all integers $i$ we have $\Haus F_i=0$ then we see that $\Haus E=0$ because the Hausdorff dimension is countably stable. Otherwise there exists an integer $i$ such that $\Haus F_i>0$. Then for any integer $n$ we see that
 	\[
 	nF_i=\{x_1+\dots+x_n:x_1,\dots,x_n\in F_i\}\subset E.
 	\]
 	Because $F_i$ is a $\times 2 \mod 1$ invariant closed subset we see that $\Haus E=1$ by Theorem \ref{ThmLMP}.
 \end{proof}

 \begin{exm}
 	We see that $\mathbb{R}/\mathbb{Z}$ is a trivial example of an ESGs in $\mathbb{R}/\mathbb{Z}$.
 \end{exm}
 
 Another less trivial example will be discussed in the next section. For now we shall discuss some further properties of ESGs. By Proposition \ref{PR01} we see that the Hausdorff dimension of an ESG in $\mathbb{R}/\mathbb{Z}$ is either zero or one. We shall see that an ESG $E$ with $\Haus E=1$ satisfies a rather strong property.
 
 \begin{defn}
 	Let $E\subset\mathbb{R}^+$. We say that $E$ is a basis of order $2$ for interval $I$ if there exists an $ x\in E$ such that
 	\[
 	I\subset E+xE.
 	\]
 \end{defn}

\begin{thm}\label{MUL}
	If $E$ is an ESG with full Hausdorff dimension then there exists $b>0$ such that $E$, viewed as a subset of $[0,1)$, is a basis of order $2$ for $(0,b]$.
\end{thm}
\begin{proof}
	In this proof use a projection trick. Here we take the addition and division operations in $\mathbb{R}$ rather than in $\mathbb{R}/\mathbb{Z}$. Consider $E$ as a subset of $[0,1]$ and $E\times E$ as a subset of $[0,1]^2$. We see that if $\Haus E=1$ then $\Haus E\times E=2$ and by the Marstrand projection theorem \cite[Theorem 3.1]{Fa2} we see that 
	$
	E+aE
	$
	has positive Lebesgue measure for all $a\in [0,1]\setminus X$ where $X\subset [0,1]$ has Hausdorff dimension $0$. In particular this means that there exists $e\in E$ such that $K_e=E+eE$ has positive Lebesgue measure.  Therefore $K_e+K_e$ contains intervals.  Due to the third property of ESG, when $x\in E$ we see that $x/2\in E$. Since $E$ is a semi-group we see that whenever $e,f\in E$ and $e,f<1/2$ it follows that $e+f\in E$. Now we see that whenever $k_1,k_2\in K_e$,
 	\[
 	\frac{k_1}{2}, \frac{k_2}{2},\frac{k_1+k_2}{2}\in K_e. 
 	\]
 	Indeed write, $k_1=e_1+ee_2, k_2=f_1+ef_2$ with $e_1,e_2,f_1,f_2\in E$. Then we see that for $i\in\{1,2\}$,
 	\[
 	\frac{e_i}{2},\frac{f_i}{2},\frac{e_i+f_i}{2}\in E.
 	\]
 	Suppose that $[a,b]\subset K_e$ with $0\leq a<b.$ Then from the discussions above we see that $[a/2,b/2]\subset K_e$. For any two intervals $I_1, I_2\subset K_e$ we see that
 	\[
 	[0.75a,0.75b]=\frac{I_1+I_2}{2}\subset K_e.
 	\]
 	We can repeat the above step. Let $N$ be an integer and $\{\epsilon_i\}_{i\in\{1,\dots,N\}}$ be any $0,1$ sequence. We see that
 	\[
 	\left[a/2+\sum_{i=1}^N \epsilon_i \frac{a}{2^{i+1}},b/2+\sum_{i=1}^N \epsilon_i \frac{b}{2^{i+1}}\right]\subset K_e.
 	\]
 	Let $N=\lceil\log \frac{2a}{b-a}\rceil$ and we see that
 	\[
 	[a/2,b]\subset K_e.
 	\]
 	The above holds for $I_1=[a/2^{k+1},b/2^{k+1}]$ and $I_2=[a/2^k,b/2^k]$ for all integers $k\geq 1$. Therefore we see that $(0,b]\subset K_e=E+eE$, which proves the result.
\end{proof}
One purpose for $(0,b)\subset E+xE$ is that  when we consider $E$ as a subset of $[0,1)$, we can often extend $E$ to $\mathbb{R}$ by considering $2^k.E=\{2^ka:a\in E\}$ and $\tilde{E}=\bigcup_{k\geq 0}2^k.E.$ Then as $(0,b)\subset E+xE$ we see that $$(0,\infty)=\tilde{E}+x\tilde{E}.$$
 
In fact we think a much stronger property should hold for ESG and this is a reason for formulating Conjecture \ref{CON1}. Notice that so far we have not used the full strength of the condition $(3)$ in Definition \ref{ER}. In fact Proposition \ref{PR01} and Theorem \ref{MUL} still hold if we replace the condition $(3)$ with the following weaker version
\[
(3'): \text{For all integer $i$, if $x\in F_i$ then $x/2\in F_i$.} 
\]  
We note that in the above statement, the $./2$ operation is taken in $[0,1)$ rather than in $\mathbb{R}/\mathbb{Z}$. The difference is that the $./2$ operation in $\mathbb{R}/\mathbb{Z}$ is not well defined and in general for $y\in\mathbb{R}/\mathbb{Z}$ the equation $2x=y$ has two solutions.
 
\section{Relation with Szemer\'{e}di's theorem} \label{RE}
 
 Recall that we identified $[0,1)$ with $\mathbb{R}/\mathbb{Z}$ and we write $+\mod 1$ to indicate that the summation is taken in $\mathbb{R}/\mathbb{Z}$. We want to show the following result about the \er set $E$, see Definition \ref{ER}.
 \begin{thm}\label{ThmER}
 	The \er set $E$ is an ESG in $\mathbb{R}/\mathbb{Z}$.
 \end{thm}
 The proof of Theorem \ref{ThmER} can be split into several lemmas. We need to check that $E$ satisfies the conditions $(1),(2),(3)$ in Definition \ref{ER}. The next lemma deals with the conditions $(1),(3)$.
 
 \begin{lma}\label{lm3}
 	There exists a sequence of closed $\times 2\mod 1$ invariant sets $F_i$ satisfying the statement of the condition $(3)$ in Definition \ref{ER} and the Erd\H{o}s set $E$ can be written as
 	\[
 	E=\bigcup_{i\in\mathbb{N}} F_i.
 	\]
 \end{lma}
 \begin{proof}
 	We write $E$ as a countable union in the following way
 	\[
 	E=\bigcup_{i\geq 3} F_i,
 	\]
 	where $F_i,i\geq 1$ are subsets of $[0,1)$ defined as follows
 	
 	$x\in F_i$ $\iff$ the binary expansion of $x$ does not contain any $i$-term arithmetic progressions of positions of digit $1$.
 	We note here that for $i\in\{1,2\}$ the sets $F_i$ are empty. Now we show that $F_i,i\geq 3$ are closed. For any integer $i\geq 3$, we consider a sequence $\{x_k\}_{k\in\mathbb{N}}\subset F_i$. Suppose that
 	\[
 	x_k\to x\in [0,1).
 	\]
 	Assume that $x\notin F_i$. Then $x$ has a $i$-term arithmetic progression of position of digit $1$ in its binary expansion. Recall that when $x$ is a dyadic rational number we have to write the binary expansion with finitely many $1's$. So in this case we can find a $i$-term progression of digits of $1's$ such that the last digit $1$ is followed by a digit $0$. This is to avoid the approximation of the following form, for example
 	\[
 	0.011111\dots=0.1.
 	\]
 	Suppose the last term of such progression is the $n$-th digit. Then if $k$ is large enough we see that $|x_k-x|\leq 2^{-n}$ and $x_k$ would also contain a $i$-term arithmetic progression. This is not possible therefore we see that $x\in F_i$ as well. Then we see that $F_i$'s are closed sets. It is clear that they are $\times 2\mod 1$ invariant and satisfy the condition $(3)$ of Definition \ref{ER}.
 
 \end{proof}
 
 Now it is left to show that the condition $(2)$ is satisfied by $E$ as well. In the following lemma, we use the decomposition $E=\bigcup_{i\in\mathbb{N}} F_i$ which appeared in the above proof.
 
 \begin{lma}\label{lm2}
 	Given any two integers $i,j\geq 3$, let $x,y$ be two sequences (real numbers in $[0,1)$), such that $x\in F_i$ and $y\in F_j$. Then there exist an integer $W(i,j)$ such that
 	\[
 	x+y\mod 1\in F_{W(i,j)}.
 	\]
 	In particular, the condition $(2)$ is satisfied by $E$.
 \end{lma}
 \begin{proof}
 	We can perform the sum $x+y\mod 1$ in $2$ steps. We first insert $1$'s of $x$ in the $0$'s of $y$. Precisely, we define the following two sequences $s,r$ of $0,1$:
 	\begin{eqnarray*}
 		&s_i=1 & \text{ if precisely one of $x_i,y_i$ is equal to $1$,}\\
 		&s_i=x_i& \text{ otherwise.}\\
 		&r_i=0& \text{ if $x_i=0, y_i=1$,}\\
 		&r_i=y_i& \text{ otherwise.}	
 	\end{eqnarray*}
 	It is clear that $x+y=s+r \mod 1$.
 	
 	After the above step, we see that $s$ is a $0,1$ sequence. Then $s$ is a concatenation of blocks of $0's$ and $1's$. Between two successive blocks of $1's$, there is at least one $0$ digit. When $s_i=0$ we see that $r_i=0$ as well. The next step is to perform $s+r$ for individual blocks of $1's$ of $s$. It is easy to see that for each individual block, the sum of $s,r$ will not disturb other blocks of $1's$ because the sum will change at most one digit $0$ next to the leftmost of a block of $1's$. For example
 	\[
 	\dots0111110\dots+ \dots0001010\dots=\dots1001000\dots
 	\]
 	
 	Let us now examine arithmetic progressions in each step. After step $1$, both $r,s$ do not contain arbitrarily long arithmetic progressions. Because the digits $1$ of $s$ come from either the digits $1$ of $x$ or $y$. The digits $1$ of $r$ are also digits $1$ of $y$. 
 	For the second step, if $x+y \mod 1$ contains arbitrarily long progressions, then there are arbitrarily long progressions on digits of $1's$ of one of the following types:
 	
 	1. the digit $1$ is an unchanged $1$ from $s$.
 	
 	2. the digit $1$ is a digit $1$ of the same position in $r$.
 	
 	3. the digit $1$ is a new $1$ resulting from step $2$ by performing the sum. This digit is next to the leftmost position of blocks of $1's$ of $s$.
 	
 	By van der Waerden's theorem, there is an integer $W(i,j)$ such that $x+y\mod 1$, when viewed as a $0,1$ sequence, does not contain any $W(i,j)$-terms arithmetic progressions of digit $1$. In fact, we can choose $W(i,j)$ to be $W(i,j,i)$ in as stated in Theorem \ref{VAN}. 
 \end{proof}
 From the above two lemmas we see that Theorem \ref{ThmER} concludes. Now we want to relate \sz's theorem with the Hausdorff dimension of $E$. 
 \begin{thm}\label{Thmsum}
 	$\text{\sz's theorem }\iff\Haus E=0\iff E+EE$  does not contain intervals.
 \end{thm}
 \begin{proof}
 	We have the decomposition $E=\bigcup_{i\geq 3} F_i$ which was mentioned before. Now each $E_i$ is closed and has lower box dimension $0$. To see the reason, for a large integer $N$, there are not so many $0,1$ sequences of length $N$ which do not contain $i$-term arithmetic progressions. In fact let $\epsilon>0$ be a small positive number. Let $N$ be a large integer and we can count the number of $0,1$ sequences of length $N$ with less than $\epsilon N$ many digit $1$. By Chernoff-Hoeffding inequality \cite{CH}, this number is 
 	\[
 	O(\exp(-D(\epsilon) N)2^N)
 	\]
 	where
 	\[
 	D(\epsilon)=\epsilon \log 2\epsilon+(1-\epsilon)\log 2(1-\epsilon)=\log 2+\epsilon\log \epsilon+(1-\epsilon)\log (1-\epsilon).
 	\]
 	Therefore we may express the bound as
 	\[
 	O(0.5^{\epsilon\log \epsilon+(1-\epsilon)\log (1-\epsilon)}).
 	\]
 	Then by \sz's theorem, we see that for large $N$, sequences over $\{0,1\}$ without $i$-term arithmetic progressions of positions of digit $1$ contain at most $o(\epsilon N)$ many digits $1$. Therefore there are at most $O(0.5^{N(\epsilon\log \epsilon+(1-\epsilon)\log (1-\epsilon)}))$ many of them. This implies that in order to cover $F_i$ with dyadic intervals of length $2^{-N}$, it is enough to use at most 
 	$O(0.5^{N(\epsilon\log \epsilon+(1-\epsilon)\log (1-\epsilon)}))$ many of them. Therefore we see that the upper box as well as the Hausdorff dimension of $E_i$ is at most
 	\[
 	\epsilon\log \epsilon+(1-\epsilon)\log (1-\epsilon).
 	\]
 	Since we can choose $\epsilon$ arbitrarily close to $0$ we see that $\Haus F_i=0$, then $\Haus E=0$ because the Hausdorff dimension is countably stable. Notice that we have in fact showed that the packing dimension of $E$ is $0$, namely, $\dim_P E=0$. Then it is easy to see that $EE$ has zero Hausdorff dimension and therefore $E+EE$ also has zero Hausdorff dimension. In particular $E+EE$ does not contain any interval.

 On the other hand, it is not hard to see that
 \[
 \Haus E=0\implies \text{ \sz's theorem}.
 \]
  Indeed, if \sz's theorem would be not true, then we can find a sequence of integers of positive upper Banach density without containing arbitrarily long arithmetic progressions. Then by a standard reduction argument, which is explained in Section \ref{seREDU}, we can assume this sequence has positive natural lower density. Then Lemma \ref{lmaDim} gives us the conclusion that $\Haus E>0$. Then by Proposition \ref{PR01} we see that
 \[
\Haus E=0\iff \Haus E\neq 1\iff \text{ \sz's theorem}.
 \]
 To conclude this theorem we need to show that
 \[
 \Haus E=1\implies E+EE \text{ contains intervals}.
 \]
  This follows easily form Theorem \ref{MUL} because $E$ is an ESG.
  \end{proof}
The key point of the above theorem is that if one can show that $E+EE$ does not contain any interval then another proof of \sz's theorem can be found. The intuition behind is that the set $E$ is constructed with numbers whose binary expansions are rather restrictive. We already know that $E+E$ cannot have positive measure because $E$ is an ESG and a Lebesgue typical number in $(0,1)$ is base $2$ normal and in particular the binary expansion has arbitrarily long consecutive digit $1$. At this stage we can say nothing about the product set $EE$. For example if $EE\subset E$ then \sz's theorem follows. However we do not know whether such a result holds and we haven't find any concrete counter examples yet. This is a reason for asking Question \ref{Q1}.

Now we shall extend the \er set to $[0,\infty)$.
\begin{defn}[Extended \er set]
	We say $x\in [0,\infty)$ is extended Erd\H{o}sian, if the binary expansion of $x$ does not contain arbitrarily long arithmetic progressions of positions of digit $1$. The collection of all Erd\H{o}sian numbers is a subset of $[0,\infty)$ and we call it the extended Erd\H{o}s set $\tilde{E}$.
\end{defn}
It is easy to check that (see the discussion below the proof of Theorem \ref{MUL})
\[
\tilde{E}=\bigcup_{k\geq 0} 2^k.E.
\]
So we see that if $\Haus E=1$ then there exists $e\in E$ such that $\tilde{E}+x\tilde{E}=[0,\infty)$. Or equivalently, we have the following result.
\begin{thm}
	\sz's theorem $\iff$ $\tilde{E}$ is not a basis of order $2$ for $[0,\infty)$.
\end{thm}

Now we want to show that Conjecture \ref{CON1} implies \sz's theorem.
\begin{thm}
	\[
	\text{Conjecture \ref{CON1}}\implies \text{ \sz's theorem}.
	\]   
\end{thm} 
\begin{proof}
Assume now Conjecture \ref{CON1}. We see that if $\Haus E=1$ then $E$ must be the trivial ESG, namely $\mathbb{R}/\mathbb{Z}$ or $[0,1)$ itself. This is impossible therefore this theorem concludes.
\end{proof}

\section{Appendix: Simple reductions of  \sz's theorem}\label{seREDU}
In this section we reduce \sz's theorem for upper Banach density to lower natural density and in a fractal geometry point of view our approach is as follows
\[
\text{Assouad dimension}\implies \text{ box dimension } \implies\text{ Hausdorff dimension }.
\]
We consider the upper Banach density version of \sz's theorem, which says that any sequence with positive upper Banach density contains arbitrarily long arithmetic progressions.
\begin{lma}\label{lm1}
	If any sequence with positive upper natural density contains arbitrarily long arithmetic progressions, then the same holds if we only require the sequence to have positive upper Banach density.
\end{lma} 
\begin{proof}
	Let $0<\rho<1, M>1$ be real numbers which can be chosen arbitrarily. Let $A\subset\mathbb{N}$ be such that the upper Banach density is $\alpha>0$. From the definition of upper Banach density we see that there is a sequence of integers $N_i\to\infty$ such that for any $i$, there exist an integer $k_i$ such that
	\[
	\#|A\cap [k_i,k_i+N_i-1]|\geq \rho\alpha N_i.
	\]
	By taking a further subsequence if necessary we can assume that $N_{i+1}>M N_i$ for all $i$. Define a sequence $w_i$ by
	\[
	w_i=\sum_{s=1}^{i} N_s.
	\]	
	Now we want to manipulate the original sequence $A$. Let $A_i=A\cap [k_i,k_i+N_i-1]$. Define a new sequence $B\subset\mathbb{N}$ by the following
	\[
	B=\bigcup_{i} (A_i-k_i+w_i).
	\]
	This sequence $B$ is constructed by shifting each $A_i$ into the interval $[w_i,w_i+N_i-1]$. The upper natural density of $B$ is bounded from below by
	\[
	\ud (B)\geq \limsup_i \frac{\# |A_i|}{\sum_{s=1}^{i} N_i}\geq \limsup_i \rho\frac{\alpha N_i}{N_i\sum_{s=1}^i M^{-s+1}}\geq  \frac{\rho M}{M-1}\alpha>0.
	\]	
	By assumption, $B$ contains arbitrarily long arithmetic progressions. We want to show that the progressions mainly lie in intervals $[w_i,w_i+N_i-1]$. 	
	Suppose there is a $k$-term progression $P$ inside $B$, where $k\geq 4$ is an integer. Then suppose the first term of $P$ is in $A_i-k_i+w_i$ for some integer $i$. 	
	Suppose the second term is also in this interval and the progressions spans over $Q$ intervals and terminate at the $(Q+1)$-th interval. Then we see that because the length of the intervals grows at least exponentially and the gap of progression is smaller than $N_i$
	\[
	M+M^2+\dots+M^Q\leq k.
	\]	
	This implies that:
	\[
	Q\leq \frac{\log (k+1)}{\log M}.
	\]
	Then in at least one of the intervals in the progression has no less than
	\[
	k/Q\geq \frac{k}{\log (k+1)}\log M
	\]
	terms. If this holds for arbitrarily large $k$ then the original sequence also contains arbitrarily long progressions.
	Similarly if the second term is in another interval say $A_{i+m}-k_{i+m}+\omega_{i+m}$, then the gap is bounded by $N_{i+m}(1+M^{-1}+\dots+M^{-m})\leq N_{i+m}\frac{M}{M-1}.$ We can argue in the same way as before. Suppose the sequence spans $Q$ intervals and stop at the $Q+1$-th interval (with a loss of factor $\frac{M-1}{M}$ but this is ignorable if $M$ is large) then we see that
	\[
	\frac{M-1}{M}(M+M^2+\dots+M^Q)\leq k.
	\]
	Then again we see that at least in one interval the progression has no less than
	\[
	\frac{k}{\log (k+1)}\log M
	\]
	terms.
	By choosing $\rho, M$ properly the upper density of the new sequence can be arbitrarily close to the original one.
\end{proof}
With the same argument we can reduce the case to lower natural density as well.
\begin{lma}\label{lm1.2}
	If any sequence with positive lower natural density contains arbitrarily long arithmetic progressions, then the same holds if we only require the sequence to have upper natural density.
\end{lma} 
\begin{proof}
	Let $0<\rho<1$ be a real numbers Assume $A$ has upper natural density $c>0$, then we can find a number $N_1>0$ such that
	\[
	\#|A\cap [1,N_1]|\geq \rho c N_1.
	\]
	We call interval $A\cap [1,N_1]$ as $I_1$.
	We can find a $N_2>2 N_1$ such that $\#|A\cap [N_1+1,N_1+N_2+1]|\geq \rho c (N_2-N_1).$ Therefore there is at least one interval of length equal to $2 N_1$ and the density of $A$ on this interval is at least $\rho c$. We now relabel $N_2=2N_1$ and call this interval $I_2$.
	We can go on finding intervals $I_k$ such that the length of $I_k$ is $2$ times that of $I_{k-1}$ and densities on each interval are bounded below by $\rho c.$ 
	Now we construct a new sequence $B$ connecting $I_k$ together. Namely,
	\[
	B_i=A_i, i\in\{1,N_1\}.
	\]
	Then we put a copy of $I_2$
	\[
	B_i=A_{i}, i\in\{N_1+1,N_1+N_2+1\}.
	\]
	The rest can be done inductively.
	Now we see that:
	\[
	\frac{\#|B\cap [1,N]|}{N}
	\]
	when $N$ are in the $k$-th translated interval $I_k$ changes between
	\[
	c\rho \text{ and } c\rho/3.
	\]
	So $B$ has positive lower natural density at least $\rho c/3$ therefore contains arbitrarily long arithmetic progressions. Then similar argument as in previous lemma leads us the conclusion. 
\end{proof}
\section{Acknowledgement}
This manuscript was written when the author was visiting Institut Mittag-Leffler during the program 'Fractal Geometry and Dynamics'.

\providecommand{\bysame}{\leavevmode\hbox to3em{\hrulefill}\thinspace}
\providecommand{\MR}{\relax\ifhmode\unskip\space\fi MR }
\providecommand{\MRhref}[2]{%
	\href{http://www.ams.org/mathscinet-getitem?mr=#1}{#2}
}
\providecommand{\href}[2]{#2}

\end{document}